\documentclass[11pt,reqno]{amsart}
\usepackage{a4wide}

\usepackage{amsthm, amssymb, amstext, amscd, amsfonts, amsxtra, latexsym, amsmath, comment, tikz, mathrsfs, mathtools, esint, stmaryrd, cancel}
\usepackage{color}
\usepackage{fancybox}
\usepackage{fullpage}
\usepackage[english]{babel}
\usepackage[latin1]{inputenc}
\usepackage{enumitem}
\usepackage{mdframed}
\usepackage{bbm}
\usepackage{url}

\numberwithin{equation}{section}
\newtheorem{theorem}{Theorem}
\numberwithin{theorem}{section}
\newtheorem{lemma}[theorem]{Lemma}
\newtheorem{proposition}[theorem]{Proposition}

\theoremstyle{definition}
\newtheorem*{definition}{Definition}
\theoremstyle{remark}
\newtheorem*{remark}{Remark}
\newtheorem*{remarks}{Remarks}

\newcommand{\Z}{\mathbb{Z}}
\newcommand{\N}{\mathbb{N}}
\newcommand{\Q}{\mathbb{Q}}

\newcommand{\R}{\mathbb{R}}
\newcommand{\C}{\mathbb{C}}

\newcommand{\im}{\operatorname{Im}}
\newcommand{\sgn}{\operatorname{sgn}}
\newcommand{\sgnstr}{\operatorname{sgn}^*}

\newcommand{\SL}{\operatorname{SL}}

\newcommand{\z}{\mathfrak{z}}
\newcommand{\x}{\mathbbm{x}}
\newcommand{\y}{\mathbbm{y}}
\renewcommand{\H}{\mathbb{H}}

\renewenvironment{proof}[1][Proof]{\begin{trivlist} \item[\hskip \labelsep {\bfseries #1:}]}{\qed\end{trivlist}}

\title{Differential operators on polar harmonic Maass forms and elliptic duality}
\author{Kathrin Bringmann}
\address{Mathematical Institute, University of Cologne, Weyertal 86-90, 50931 Cologne, Germany}
\email{kbringma@math.uni-koeln.de}
\author{Paul Jenkins}
\address{Brigham Young University, Department of Mathematics, Provo, UT 84602 USA}
\email{jenkins@math.byu.edu}
\author{Ben Kane}
\address{Department of Mathematics, University of Hong Kong, Pokfluam, Hong Kong}
\email{bkane@hku.hk}
\date{\today}
\begin{document}
\begin{abstract}
In this paper, we study polar harmonic Maass forms of negative integral weight. Using work of Fay, we construct Poincar\'e series which span the space of such forms and show that their elliptic coefficients exhibit duality properties which are similar to the properties known for Fourier coefficients of harmonic Maass forms and weakly holomorphic modular forms. 
\end{abstract}
\thanks{The research of the first author is supported by the Alfried
	Krupp Prize for Young University Teachers of the Krupp foundation and
	the research leading to these results receives funding from the
	European Research Council under the European Union's Seventh Framework
	Programme (FP/2007--2013) / ERC Grant agreement n. 335220 - AQSER.  This work was partially supported by a grant from the Simons Foundation (\#281876 to Paul Jenkins).  The research of the third author was supported by grants from the Research Grants Council of the Hong Kong SAR, China (project numbers HKU 27300314, 17302515, and 17316416).}
\maketitle
\section{Introduction and statement of results}

Harmonic Maass forms are smooth functions on the upper half-plane $\mathbb{H}$ which are annihilated by the hyperbolic Laplacian and have at most linear exponential growth at the cusps. These naturally generalize \begin{it}weakly holomorphic modular forms\end{it}, which are meromorphic modular forms whose only poles appear at cusps; in particular, holomorphicity in $\H$ is replaced with annihilation by the hyperbolic Laplacian defined in~(\ref{laplace}).  Harmonic Maass forms play a major role in work on mock theta functions, singular moduli and their real quadratic analogues, and many other applications.  Functions with properties similar to harmonic Maass forms, but with poles in the upper half-plane, appear in a number of recent results, including work on the resolvent kernel~\cite{Fa}, outputs of theta lifts~\cite{Borcherds}, cycle integrals~\cite{BKvP}, Fourier coefficients of meromorphic forms~\cite{BKFourier, Pe1}, and the computation of divisors of modular forms \cite{BKLRO}.  In considering functions with poles in the upper half-plane rather than solely at the cusps, we might expect that the behavior of such functions is similar to the behavior of harmonic Maass forms.  In this paper we study spaces of such \emph{polar harmonic Maass forms}, which generalize harmonic Maass forms in the same way that meromorphic modular forms generalize weakly holomorphic modular forms.

From another perspective, the subspace of polar harmonic Maass forms consisting of meromorphic modular forms is analogous to the subspace of harmonic Maass forms consisting of weakly holomorphic modular forms.
Meromorphic modular forms have not only a Fourier expansion at the cusp $i\infty$, but also an \begin{it}elliptic expansion\end{it}
\begin{equation}\label{eqn:fellexp}
f(z)=
(z-\overline{\varrho})^{-2\kappa}
\sum_{n\gg -\infty} c_{f,\varrho}(n) X_{\varrho}^n(z)
\end{equation}
around each point $\varrho \in \H$ in terms of powers of $X_{\varrho}(z) := \frac{z-\varrho}{z-\overline{\varrho}}$, where $c_{f,\varrho}(n)\in\C$.  Polar harmonic Maass forms have a more general elliptic expansion where the coefficients $c_{f,\varrho}(n)$ may additionally depend on $r_{\varrho}(z):=|X_{\varrho}(z)|$; we give this expansion in Proposition~\ref{prop:elliptic} below.  Here and throughout, $\kappa$ is assumed to be an arbitrary integer, and we use $k$ instead if there is some restriction on the weight (for instance, if we require $k \in \N$).

In addition to having similar elliptic expansions, polar harmonic Maass forms and meromorphic modular forms of weights $2\kappa$ and $2-2\kappa$ are interconnected by certain differential operators which naturally occur in the theory of harmonic Maass forms. For $\kappa\in\Z$ and $k\in\N$, set
\begin{equation}\label{eqn:xiDdef}
\xi_{2\kappa} := 2iy^{2\kappa}\overline{\frac{\partial}{\partial\overline{z}}},\qquad  D^{2k-1}:=\left(\frac{1}{2\pi i}\frac{\partial}{\partial z}\right)^{2k-1},
\end{equation}
where $z=x+iy\in\H$.  If $f$ satisfies weight $2\kappa$ modularity, then $\xi_{2\kappa} (f)$ is modular of weight $2-2\kappa$, while if $f$ satisfies weight $2-2k$ modularity, then $D^{2k-1}(f)$ satisfies weight $2k$ modularity.  

Given a polar harmonic Maass form of weight $2\kappa$, one may eliminate the singularity at $i\infty$ and, if $\kappa \geq 0$, also the constant term in the Fourier expansion, by subtracting an appropriate harmonic Maass form
(cf.~\cite[Theorem 6.10]{MockBook} for the existence of forms with arbitrary principal parts).  This yields a weight $2\kappa$ \begin{it}polar harmonic cusp form\end{it}, a weight $2\kappa$ polar harmonic Maass form which vanishes at $i\infty$ if $\kappa \in \N_0$ and is bounded at $i\infty$ if $\kappa\in-\N$.
 We denote the subspace of such forms  by $\mathscr{H}_{2\kappa}$.  A canonical basis for this space may be defined by specifying the growth behavior near singularities in $\SL_2(\Z)\backslash\H$, which is given via principal parts at $\z$; see~\eqref{eqn:elliptic} for further details on the principal parts which may occur.
This basis is defined in~\eqref{eqn:Pdef} below, and we show in Theorem~\ref{OperatorTheorem} that for $k\in\N_{>1}$ they indeed span $\mathscr{H}_{2-2k}$.  Specifically, for each $n\in-\N$ and $\z\in\H$, in \eqref{eqn:Pdef} we construct  the unique weight $2-2k$ polar harmonic cusp form $\mathbb{P}_{2-2k, n}^\z$ with principal part
\begin{equation}\label{eqn:n<0sing}
2\omega_{\z}\mathcal{C}_{2k-1,-n}(z-\overline{\z})^{2k-2} X_{\z}^{n}(z).
\end{equation}
Here $\omega_{\z}$ is the size of the stabilizer of $\mathfrak{z}$ in $\operatorname{PSL}_2(\mathbb{Z})$ and $\mathcal{C}_{2k-1,-n}$ is the constant defined in \eqref{eqn:Cdef} below and explicitly computed as a quotient of factorials in \eqref{eqn:n<0singconst}. For $n\in\N_0$, the functions $\mathbb{P}_{2-2k,n}^{\z}$ have non-meromorphic principal parts.  We describe these explicitly in Theorem \ref{thm:PoincProperties} below.

To understand the behavior of the basis elements $\mathbb{P}_{2-2k,n}^{\z}\in \mathscr{H}_{2-2k}$ under the differential operators defined in \eqref{eqn:xiDdef}, we define the subspace $\mathbb{S}_{2\kappa}\subseteq\mathscr{H}_{2\kappa}$ consisting of meromorphic modular forms without poles at $i\infty$, which we call \emph{meromorphic cusp forms}. For each point $\z$ in the fundamental domain $\SL_2(\Z)\backslash \mathbb{H}$, we let $\mathscr{H}_{2\kappa}^{\z}$ (resp. $\mathbb{S}_{2\kappa}^\z$) be the subspace of forms in $\mathscr{H}_{2\kappa}$ (resp. $\mathbb{S}_{2\kappa}$) with singularities allowed only at $\z$.  For $k\in\N_{>1}$, Petersson (cf. \cite[equation (5c.3)]{Pe1} or \cite[equation (21)]{Pe2}) defined a family of meromorphic Poincar\'{e} series $\Psi_{2k, n}^\z$ which form a natural canonical basis for the space $\mathbb{S}_{2k}^\z$.  Specifically, for $n \in-\N$ and a point $\z \in \H$, the function $\Psi_{2k, n}^\z$ is the unique meromorphic cusp form which is orthogonal to cusp forms (see \cite[Satz 8]{Pe2}) under a regularized inner product defined in \cite[equation (3)]{Pe2} and whose principal part is
\begin{equation}\label{eqn:PsiPP}
2\omega_{\z}
\left(z-\overline{\z}\right)^{-2k}
 X_{\z}^{n}(z).
\end{equation}

As shown in the next theorem, the action of the differential operators $\xi_{2-2k}$ and $D^{2k-1}$ give an additional natural splitting of the space $\mathbb{S}_{2k}^\z$ into three subspaces, which we denote by $\mathbb{D}_{2k}^{\z}$, $\mathbb{E}_{2k}^{\z}$, and the space of cusp forms $S_{2k}$.  Again using the regularization \cite[equation (3)]{Pe2}, or its extension \cite[equation (3.3)]{BKvP} to arbitrary meromorphic cusp forms, the subspace $\mathbb{D}_{2k}^{\z}$ (resp. $\mathbb{E}_{2k}^{\z}$) consists of those forms in $\mathbb{S}_{2k}^{\z}$ which are orthogonal to cusp forms and whose principal parts are linear combinations of \eqref{eqn:PsiPP} with $n\leq -2k$ (resp. $-2k<n<0$). The families $\Psi_{2k, m}^\z$ of meromorphic Poincar\'{e} series with $m \leq -2k$ or with $-2k < m < 0$ form bases for $\mathbb{D}_{2k}^{\z}$ and $\mathbb{E}_{2k}^{\z}$ respectively.
\begin{theorem}\label{OperatorTheorem}
Suppose that $k\in\N_{>1}$.
\noindent

\noindent
\begin{enumerate}[leftmargin=*,label={\rm(\arabic*)}]
\item
Every $F\in \mathscr{H}_{2-2k}$ is a linear combination of the functions from $\{\mathbb{P}_{2-2k,n}^{\z}: \z\in\H, n\in\Z\}.$ Moreover, if the only poles of $F$ in $\H$ occur at points equivalent to $\z$ under the action of $\SL_2(\Z)$, then $F$ is a linear combination of functions from $\{\mathbb{P}_{2-2k,n}^{\z}: n\in\Z\}.$
\item If $F\in \mathscr{H}_{2-2k}^{\z}$, then $\xi_{2-2k}(F)\in \mathbb{D}_{2k}^{\z}$ {\rm(}resp. $\xi_{2-2k}(F)\in S_{2k}${\rm)} if and only if $D^{2k-1}(F)\in S_{2k}$ {\rm(}resp. $D^{2k-1}(F)\in \mathbb{D}_{2k}^{\z}${\rm).}
\item If $F\in \mathscr{H}_{2-2k}^{\z}$, then $\xi_{2-2k}(F)\in \mathbb{E}_{2k}^{\z}$ if and only if $D^{2k-1}(F)\in \mathbb{E}_{2k}^{\z}.$
\item  For $n\in\Z$ and $\z=\x+i\y\in\H$, we have
\begin{align*}
\xi_{2-2k}\left(\mathbb{P}_{2-2k,n}^{\z}\right)&=(4\y)^{2k-1}\Psi_{2k, -n-1}^{\z},\\
D^{2k-1}\left(\mathbb{P}_{2-2k,n}^{\z}\right)&=(2k-2)!\left(-\frac{\y}{\pi}\right)^{2k-1}\Psi_{2k, n+1-2k}^{\z}.
\end{align*}
\rm
\end{enumerate}
\end{theorem}
\begin{remark}
The three subspaces in the splitting explained by Theorem \ref{OperatorTheorem} (2) and Theorem \ref{OperatorTheorem} (3) are analogues of certain subspaces of weakly holomorphic modular forms. Specifically, the
space $\C E_{2k}$ spanned by the Eisenstein series is paired with itself in the same way,
while the space $S_{2k}$ of cusp forms is paired with its orthogonal complement inside the subspace of
weakly holomorphic modular forms which have vanishing constant terms in their Fourier expansion.
\end{remark}

For $\z_1,\z_2\in \H$, our next result gives a duality-type relationship between the coefficients of $\Psi_{2k, m}^{\z_1}$ and those of $\mathbb{P}_{2-2k, \ell}^{\z_2}$.  To state it, for $m \in-\N$, let $c_{2k, \z_2}^{\z_1}(m, n)$ denote the $n$th coefficient in the elliptic expansion \eqref{eqn:fellexp} around $\varrho=\z_2$ of $\Psi_{2k, m}^{\z_1}$.  Similarly, $c_{2-2k, \z_1}^{\z_2, +}(m, n)$ is the $n$th coefficient in the meromorphic part of the elliptic expansion (see \eqref{eqn:elliptic}) around $\varrho=\z_1$ of $\mathcal{C}_{2k-1,-m}^{-1}\mathbb{P}_{2-2k, m}^{\z_2}$; in other words, by~\eqref{eqn:n<0sing} these are the coefficients of the unique weight $2-2k$ polar harmonic Maass forms with principal parts
\[
2\omega_{\z}(z-\overline{\z})^{2k-2} X_{\z}^{m}(z),
\]
which closely resemble the principal parts \eqref{eqn:PsiPP} in positive weight.
\begin{theorem}\label{DualityTheorem}
For $m, n\in\N_0$, we have
\[
c_{2-2k, \z_1}^{\z_2, +}(-m-1, n) = -c_{2k, \z_2}^{\z_1} (-n-1, m).
\]
\end{theorem}
\begin{remark}
Similar results for Fourier coefficients are well known.  Petersson used such identities in his construction of a basis of meromorphic modular forms (see \cite[(3a.9)]{Pe1}, while a systematic study of them originated from Zagier's work on singular moduli \cite{Zagier}.  In the interim, results have been obtained by a number of authors, including 
the second author and Duke in \cite{DukeJenkins}, and Guerzhoy \cite{GuerzhoyGrids}, among others.  To give one such result, for $m,n\in\N$, take $\z_1=\z_2=i\infty$ and let $c_{2k}(-m,n)$ denote the $n$th coefficient of the weight $2k$ weakly holomorphic modular form which grows towards $i\infty$ like $e^{-2\pi i m z}$ and let $c_{2-2k}^+(-m,n)$ be the $n$th coefficient of the holomorphic part of the weight $2-2k$ harmonic Maass form which grows towards $i\infty$ like $e^{-2\pi i m z}$.  Then one has 
\[
c_{2-2k}^+(-n,m)=-c_{2k}(-m,n).
\]
\end{remark}

The paper is organized as follows.  In Section \ref{sec:prelim}, we introduce polar harmonic Maass forms and recall results from Fay, who studied related functions in \cite{Fa}.  In Section \ref{sec:ellexp}, we relate Fay's functions to polar harmonic Maass forms and compute the elliptic expansions of polar harmonic Maass forms, and in Section~\ref{sec:Poincare} we investigate Poincar\'e series and prove Theorem \ref{OperatorTheorem}. We conclude the paper by proving Theorem \ref{DualityTheorem} in Section \ref{sec:duality}.

\section{Preliminaries}\label{sec:prelim}
\subsection{Basic definitions}
For $M=\begin{psmallmatrix} a&b\\c&d\end{psmallmatrix}\in\operatorname{SL}_2(\mathbb{Z})$, $\kappa \in \Z$, and $f:\mathbb{H}\to\mathbb{C}$, we define the usual slash operator
$$
f|_{2\kappa}M(z)=f|_{2\kappa,z}M(z):=(cz+d)^{-2\kappa}
f(Mz).
$$
\begin{definition}
For $\kappa\in\mathbb{Z}$, a {\it polar harmonic Maass form of weight $2\kappa$} is a function $F:\mathbb{H}\to\mathbb{C}$ which is real analytic outside a discrete set of points and satisfies the following conditions:
\begin{enumerate}[leftmargin=*,label={\rm(\arabic*)}]
\item For every $M\in\operatorname{SL}_2(\mathbb{Z})$, we have $F|_{2\kappa}M=F$.
\item We have $\Delta_{2\kappa}(F)=0$, with $\Delta_{2\kappa}$ the {\it weight $2\kappa$ hyperbolic Laplace operator}
\begin{align}
\Delta_{2\kappa}:=-y^2\left(\frac{\partial^2}{\partial x^2}+\frac{\partial^2}{\partial y^2}\right)+2i\kappa y\left(\frac{\partial}{\partial x}+i\frac{\partial}{\partial y}\right).\label{laplace}
\end{align}
\item For every $\mathfrak{z}\in\mathbb{H}$, there exists $n\in\mathbb{N}_0$ such that $(z-\mathfrak{z})^nF(z)$ is bounded in some neighborhood of $\mathfrak{z}$.
\item The function $F$ has at most linear exponential growth at $i\infty$; that is,  $F(z) = O(e^{Cy})$ for some constant $C\in\R^+$ (uniform in $x$ for $y$ sufficiently large) as $y \rightarrow \infty$.
\end{enumerate}
If $(2)$ is replaced by $\Delta_{2\kappa}(F)=\lambda F$, then $F$ is called a {\it polar Maass form} with eigenvalue $\lambda$.

\end{definition}
Denote by $\mathcal{H}_{2\kappa}$ the space of polar harmonic Maass forms of weight $2\kappa$.  The subspace of $\mathcal{H}_{2\kappa}$ consisting of forms that map under $\xi_{2\kappa}$ to cusp forms is denoted by $\mathcal{H}_{2\kappa}^{\operatorname{cusp}}$; more generally, we add the superscript ``cusp'' to any subspace of $\mathcal{H}_{2\kappa}$ to indicate the space formed by taking the intersection of the subspace with $\mathcal{H}_{2\kappa}^{\operatorname{cusp}}$.  We also use the superscript $\z$ to indicate the subspace of forms whose only singularity in $\SL_2(\Z)\backslash\H$ appears at $\z$.

Although in this paper we are primarily interested in expansions of polar harmonic Maass forms around points in the upper half-plane, for completeness and for later comparison we next recall some properties about the Fourier expansions of polar harmonic Maass forms around $i\infty$.  These expansions yield natural decompositions of polar harmonic Maass forms into holomorphic and non-holomorphic parts (cf. \cite[Proposition 4.3]{Hejhal}).  Namely, for a polar harmonic Maass form $F$ of weight $2-2k<0$ and $y\gg_F 1$, we have
$$
F(z)=F^+(z)+F^-(z)
$$
where, for some $c_F^\pm(n)\in\mathbb{C}$, we define the {\it holomorphic part $F^+$} (resp. {\it non-holomorphic part $F^-$}) of $F$ at $i\infty$ as
\begin{align}
F^+(z):&=\sum_{n\gg -\infty}c^+_F(n)e^{2\pi i nz},\nonumber\\
F^-(z):&=c^-_F(0)y^{2k-1}+\sum_{\substack{n\ll\infty \\ n\neq 0}}c^-_F(n)\Gamma(2k-1, -4\pi ny)e^{2\pi i nz},
\label{nonhol}
\end{align}
with the incomplete gamma function $\Gamma(\alpha, w):=\int_{w}^{\infty}e^{-t}t^{\alpha-1}dt$.  The sum of all of the terms which grow towards $i\infty$ is called the {\it principal part} of $F$.

We next consider elliptic expansions of polar harmonic Maass forms.  Rather than expansions in $e^{2\pi i z}$, the natural expansions of polar harmonic Maass forms around $\varrho$ are given in terms of $X_{\varrho}(z)$.  We further write
\begin{equation}\label{eqn:rdef}
r_{\varrho}(z):=\tanh\!\left(\frac{d(z,\varrho)}{2}\right)=\left|X_{\varrho}(z)\right|,
\end{equation}
with $d(z,\varrho)$ the hyperbolic distance between $z$ and $\varrho$.  The second identity in the definition of $r_{\varrho}(z)$ follows by the well-known formula (see \cite[p. 131]{Beardon})
\begin{equation}\label{d}
\cosh(d(z,\varrho))=1+\frac{|z-\varrho|^2}{2y\eta},
\end{equation}
where throughout the paper $\eta:=\im(\varrho)$. From \eqref{d}, for $M\in \SL_2(\Z)$ one also immediately obtains the invariance
\begin{equation}\label{eqn:dinv}
d(Mz,M\varrho) = d(z,\varrho).
\end{equation}
For $0\leq {Z}<1$ and $a\in \N$ and $b\in \Z$, we also require the function
\begin{equation}\label{eqn:beta0def}
\beta_0\left({Z};a,b\right):=\beta\left({Z};a,b\right)-\mathcal{C}_{a,b}
\end{equation}
where
\begin{equation*}
\beta({Z};a,b):=\int_0^{Z} t^{a-1} (1-t)^{b-1} dt
\end{equation*}
 is the \begin{it}incomplete beta function\end{it} and
\begin{equation}\label{eqn:Cdef}
\mathcal{C}_{a,b}:=\sum_{\substack{0\leq j\leq a-1 \\ j\neq -b}} \binom{a-1}{j}\frac{(-1)^{j}}{j+b}.
\end{equation}
 Note that by \cite[8.17.7]{NIST}, we have
\begin{equation}\label{bF}
\beta(Z;a,b)=\frac{Z^a}{a}{}_2F_1(a, 1-b; a+1;Z),
\end{equation}
where ${}_2F_1$ is the {\it Gauss hypergeometric function} defined by
$$
{}_2F_1(a,b;c;Z) := \sum_{n=0}^{\infty} \frac{(a)_n (b)_n}{(c)_n} \frac{Z^n}{n!}
$$
with $(a)_n:= a(a+1)\cdots(a+n-1)$. We often use the fact that
\begin{equation}\label{eqn:2F10}
{_2F_1}(a,0;c;Z)={_2F_1}(0,b;c;Z)=1.
\end{equation}
We also require the {\it Euler transformation} (see 15.8.1 of \cite{NIST})
\begin{equation}\label{Euler}
{}_2F_1(a,b;c;Z)=(1-Z)^{c-a-b}{}_2F_1(c-a, c-b;c;Z).
\end{equation}

The modified incomplete $\beta$-function $\beta_0$ may also be written in special cases as a hypergeometric function, as can be seen by a direct calculation.
\begin{lemma}\label{bet0lem}
 Assume that $0\leq Z <1$, $a\in \N$, and $b\in \Z$.
 \begin{enumerate}[leftmargin=*,label={\rm(\arabic*)}]
 \item We have
$$
 \beta_0 (Z;a,b) = \sum_{ \substack{0\leq j \leq a-1 \\ j\neq -b}} \binom{a-1}{j} \frac{(-1)^{
j+1
}}{j+b} (1-Z)^{j+b} +\delta_{1-a\leq b \leq 0}
 \binom{a-1}{-b} (-1)^{
b+1
} \log (1-Z) .
$$
Here and throughout we use the notation $\delta_S=1$ if some property $S$  is true and $0$ otherwise.

 \item If $b>0$, then
 \[
 \beta_0 (Z;a,b) = -\frac{1}{b} (1-Z)^b {}_2F_1 (b,1-a;1+b ;1-Z)=-\beta(1-Z;b,a).
 \]
 \end{enumerate}
\end{lemma}

\indent
We have the following elliptic expansion of weight $2-2k$ harmonic functions, whose proof is deferred to Section \ref{sec:ellexp}.
\begin{proposition}\label{prop:elliptic}
Suppose that $k\in \N$ and $\varrho \in\mathbb{H}$.
\begin{enumerate}[leftmargin=*,label={\rm(\arabic*)}]
\item If $F$ satisfies $\Delta_{2-2k} (F) = 0$ and for some $n_0\in \N$ the function $r_{\varrho}^{n_0}(z)F(z)$ is bounded in some neighborhood $\mathcal{N}$ around $\varrho$, then there exist $c_{F,\varrho}^{\pm}(n) \in\mathbb{C}$, such that for $z\in\mathcal{N}$ and $n_{k}:=\min\!\left(2k-2, n_0\right)$, we have
\begin{multline}\label{eqn:elliptic}
F(z)=
\left(z-\overline{\varrho}\right)^{2k-2}
\Bigg(\sum_{n\geq -n_0} c_{F,\varrho}^+(n) X_{\varrho}^n(z)+ \sum_{ n=0}^{n_{k}} c_{F,\varrho}^-(n)\beta\left(1-r_{\varrho}^2(z);2k-1,-n\right) X_{\varrho}^n(z)\\
+ \sum_{\substack{n\leq n_0\\ n\notin [0,n_{k}]}} c_{F,\varrho}^-(n)\beta_0\left(1-r_{\varrho}^2(z);2k-1,-n\right) X_{\varrho}^n(z)\Bigg).
\end{multline}
\item If $F\in \mathcal{H}_{2-2k}$, then the sum in \eqref{eqn:elliptic} only runs over those $n$ which satisfy $n\equiv k-1\pmod{\omega_{\varrho}}$.
If  $F\in\mathcal{H}_{2-2k}^{\operatorname{cusp}}
$, then the second sum is empty and the third sum only runs over $n<0$.
\end{enumerate}
\end{proposition}
\begin{remark}
Instead of the expansion given in \eqref{eqn:elliptic}, one could rewrite the second sum in the shape of the third to get a seemingly more uniform expansion.  However, it is natural to split off these terms because they have logarithmic singularities.  They are also special, as we shall see in Proposition \ref{Explicit}, in that they are annihilated neither by $\xi_{2-2k}$ nor $D^{2k-1}$.  Thus, they may be viewed in a sense both as both meromorphic and non-meromorphic parts.  This emulates the constant term of the non-holomorphic part \eqref{nonhol} of the expansion at $i\infty$, which is a constant multiple of $y^{2k-1}$, is annihilated by neither operator, and also exhibits a logarithmic singularity.
\end{remark}

For $F$ annihilated by $\Delta_{2-2k}$ (with $k\in\N$), we define the \begin{it}meromorphic part of the elliptic expansion\end{it} \eqref{eqn:elliptic} around $\varrho$ by
$$
F_{\varrho}^+ (z):=
\left(z-\overline{\varrho}\right)^{2k-2}
\sum_{n\geq -n_0}c_{F,\varrho}^+(n)X_{\varrho}^n(z)
$$
and its \begin{it}non-meromorphic part\end{it} by
\begin{multline*}
F_{\varrho}^- (z) :=
\left(z-\overline{\varrho}\right)^{2k-2}
\sum_{n=0}^{n_{k}} c_{F,\varrho}^-(n)\beta\left(1-r_{\varrho}^2(z);2k-1,-n\right) X_{\varrho}^n(z)
\\
+\left(z-\overline{\varrho}\right)^{2k-2}
\sum_{\substack{n\leq n_0\\ n\notin [0,n_{k}]} } c_{F,\varrho}^-(n)\beta_0\left(1-r_{\varrho}^2(z);2k-1,-n\right) X_{\varrho}^n(z).
\end{multline*}

The next proposition, proven in Section~\ref{sec:ellexp}, explicitly gives the elliptic expansion under the action of the operators $\xi_{2-2k}$ and $D^{2k-1}$.
\begin{proposition}\label{Explicit}
For $k\in \N$ and $F:\mathbb{H} \to \mathbb{C}$ satisfying $\Delta_{2-2k}(F) = 0$, we have
\begin{align*}
	\xi_{2-2k}\left(F(z)\right)&=(4\eta)^{2k-1}
(z-\overline{\varrho})^{-2k}
\sum_{n\leq n_0}\overline{c_{F,\varrho}^{-}(n)} X_\varrho^{-n-1}(z),\\
	D^{2k-1}\left(F(z)\right)&=\left(\frac{\eta}{\pi}\right)^{2k-1}
(z-\overline{\varrho})^{-2k}
\sum_{n\geq -n_0}b_{F,\varrho}(n)X_{\varrho}^{n+1-2k}(z)
\end{align*}
and
\[
b_{F,\varrho}(n):=\begin{cases}
-\frac{(-n+2k-2)!}{(-n-1)!}c^+_{F,\varrho}(n) &\text{if }n<0,\\
-(2k-2)!c^-_{F, \varrho}(n)&\text{if } 0\leq n\leq n_{k},\\
\frac{n!}{(n+1-2k)!}c^+_{F,\varrho}(n)&\text{if } n\geq 2k-1.
\end{cases}
\]
\end{proposition}
In addition to the operators $\xi_{2\kappa}$ and $D^{2k-1}$ given in \eqref{eqn:xiDdef}, we require the classical \textit{Maass raising} and \textit{lowering operators}:
\begin{equation*}
	R_{2\kappa}:=2i\frac{\partial}{\partial z}+\frac{
2\kappa
}{y}\qquad\text{and}\qquad L:=-2iy^2\frac{\partial}{\partial \overline{z}}.
\end{equation*}
The raising operator (resp. lowering operator) increases (resp. decreases) the weight by $2$. Moreover
\begin{equation}\label{DLR}
-\Delta_\kappa=L\circ R_\kappa+\kappa=R_{\kappa-2}\circ L.
\end{equation}
We also require iterated raising
$$
R^n_\kappa:=R_{\kappa+2(n-1)}\circ\cdots\circ R_{\kappa+2}\circ R_\kappa.
$$
For $k\in\N$, the raising operator and $D^{2k-1}$ are related by {\it Bol's identity}
\begin{equation}\label{Bol}
D^{2k-1}=(-4\pi)^{1-2k}R^{2k-1}_{2-2k}.
\end{equation}
\subsection{Work of Fay}
In this section we recall work of Fay \cite{Fa} and rewrite some of his statements in the notation used in this paper. Fay considered functions $g: \mathbb{H} \to \mathbb{C}$ transforming for $M=\begin{psmallmatrix}a&b\\c&d\end{psmallmatrix}\in\operatorname{SL}_2(\mathbb{Z})$ as
$$
g(Mz)=\left(\frac{cz+d}{c\overline{z}+d}\right)^\kappa g(z).
$$
Then $f(z):=y^{-\kappa}g(z)$ transforms as
\begin{equation*}
f(Mz)=(cz+d)^{2\kappa}f(z).
\end{equation*}

Define the operator
$$
\mathcal{D}_\kappa:=y^2\left(\frac{\partial^2}{\partial x^2}+\frac{\partial^2}{\partial y^2}\right)-2i\kappa y
\frac{\partial}{\partial x}.
$$
By \cite[page 144]{Fa}, for $M=\left(\begin{smallmatrix} a& b\\ c&d\end{smallmatrix}\right)\in\operatorname{SL}_2(\mathbb{R})$,
we have
$$
\mathcal{D}_\kappa\left(\left(\frac{c\overline{z}+d}{cz+d}\right)^\kappa g(Mz)\right)=\left(\frac{c\overline{z}+d}{cz+d}\right)^\kappa\left[\mathcal{D}_\kappa\left(g(w)\right)\right]_{w=Mz}.
$$
Let $\mathcal{F}_{\kappa,s}$ denote the space of $g:\mathbb{H}\to\mathbb{C}$ satisfying the following conditions:
\begin{enumerate}[leftmargin=*,label={\rm(\arabic*)}]
\item $g(Mz)=\left(\frac{cz+d}{c\overline{z}+d}\right)^\kappa g(z)$;
\item $\mathcal{D}_\kappa(g)=s(s-1)g$;
\item $g$ has at most finitely many singularities of finite order in $\operatorname{SL}_2(\mathbb{Z})\backslash \overline{\mathbb{H}}$, where $\overline{\H}:=\H\cup \Q\cup\{i\infty\}$.
\end{enumerate}

Functions in $\mathcal{F}_{\kappa,s}$ are closely related to polar Maass forms.
In order to study the relationship between $\mathcal{D}_{\kappa}$ acting on Fay's functions and $\Delta_{2\kappa}$ acting on polar Maass forms,
we require the following variants of the Maass raising and lowering operators (see \cite[(3)]{Fa}),
 \begin{equation*}
	\mathcal{K}_\kappa=\mathcal{K}_{\kappa,z}
:=2iy\frac{\partial}{\partial z}+\kappa, \qquad\qquad \mathcal{L}_\kappa=
\mathcal{L}_{\kappa,z}:=-2iy\frac{\partial}{\partial \overline{z}}-\kappa.
\end{equation*}
Note that $\mathcal{K}_\kappa$ sends $\mathcal{F}_{\kappa, s}$ to $\mathcal{F}_{\kappa+1, s}$ and $\mathcal{L}_\kappa$ sends $\mathcal{F}_{\kappa, s}$ to $\mathcal{F}_{\kappa-1, s}$.
Moreover (see \cite[(7)]{Fa})
\begin{equation}\label{DLK}
\mathcal{D}_\kappa=\mathcal{L}_{\kappa+1}\circ\mathcal{K}_\kappa+\kappa(1+\kappa)=\mathcal{K}_{\kappa-1}\circ\mathcal{L}_\kappa+\kappa(\kappa-1).
\end{equation}
We also require iterated raising and lowering
\begin{align*}
	\mathcal{K}^n_\kappa:=\mathcal{K}_{\kappa+n-1}\circ\cdots\circ\mathcal{K}_{\kappa+1}\circ\mathcal{K}_{\kappa},\
	\mathcal{L}^n_\kappa:=\mathcal{L}_{\kappa+n-1}\circ\cdots\circ\mathcal{L}_{\kappa+1}\circ\mathcal{L}_{\kappa}.
\end{align*}

We next translate these operators into the notation used in this paper and compare eigenfunctions under these operators.
\begin{proposition}\label{prop:LRKL}
\noindent

\noindent
\begin{enumerate}[leftmargin=*,label={\rm(\arabic*)}]
\item For $n\in \N_0$, we have
\begin{align}\label{eqn:KRrep}
\mathcal{K}_\kappa^{n}\left(g(z)\right)&=y^{\kappa+n} R_{2\kappa}^{n}\left(f(z)\right),
\\
\label{eqn:LLrep}
\mathcal{L}_\kappa^{n}\left(g(z)\right)&=y^{\kappa-n}L^{n}\left(f(z)\right),
\\
\label{DDel}
\mathcal{D}_\kappa\left(g(z)\right)&=-y^\kappa\left(\Delta_{2\kappa}+\kappa(1-\kappa)\right)f(z).
\end{align}
If $g\in \mathcal{F}_{\kappa,s}$, then
\begin{equation}\label{eqn:Delta2kf}
\Delta_{2\kappa}(f)= (s-\kappa)(1-s-\kappa)f.
\end{equation}
In particular, $f$ is harmonic if and only if $g\in \mathcal{F}_{\kappa,\kappa}$ or $g\in \mathcal{F}_{\kappa,1-\kappa}$.

\item
The function $g\in \mathcal{F}_{\kappa,s}$ if and only if
the function $f$ is a polar Maass form of weight $2\kappa$ with eigenvalue $(s-\kappa)(1-s-\kappa)$.  In particular if $g\in \mathcal{F}_{\kappa,\kappa}$ or $g\in \mathcal{F}_{\kappa,1-\kappa}$ and grows at most like $y^{\kappa}$ for $y\to \infty$, then $f\in \mathscr{H}_{2\kappa}$.
\end{enumerate}
\end{proposition}
\begin{proof}
(1)
Firstly it is not hard to see that
\begin{equation}\label{KR}
	\mathcal{K}_\kappa\left(g(z)\right)=y^{\kappa+1} R_{2\kappa}(f(z)).
\end{equation}
Iterating \eqref{KR} yields \eqref{eqn:KRrep}.
Similarly, to prove \eqref{eqn:LLrep}, one first shows that
\begin{equation}\label{LL}
\mathcal{L}_\kappa\left(g(z)\right)=y^{\kappa-1}L(f(z)).
\end{equation}
One then obtains \eqref{eqn:LLrep} inductively.
The eigenfunction property \eqref{DDel} then follows using \eqref{DLK}, \eqref{LL}, \eqref{KR}, and \eqref{DLR}.
To prove \eqref{eqn:Delta2kf}, suppose that
 $g\in \mathcal{F}_{\kappa,s}$.  Then, by \eqref{DDel}, we have
\[
0=\mathcal{D}_{\kappa}(g(z))-s(s-1)g(z)=-y^{\kappa}\Delta_{2\kappa}\left(f(z)\right) -\kappa(1-\kappa)y^{\kappa}f(z) - s(s-1)y^{\kappa} f(z).
\]

(2)  Part (1) implies that the eigenfunction properties of $f$ and $g$ are equivalent. Comparing the singularities of both functions then yields the claim.
\end{proof}

\indent Fay then considered a natural family of functions which behave well under his differential operators when multiplied by $e^{in\theta_{\z}(z)}$ with $\theta_{\z}(z)\in\R$ satisfying $X_{\z}(z)=r_{\z}(z) e^{i\theta_{\z}(z)}$.  For $s\in\C$, $\kappa\in \R$, and $z,\mathfrak{z}\in \H$, these are given by (see \cite[p. 147]{Fa}, slightly modified)
\[
\mathcal{P}_{s,\kappa}^n(z,\z):=\widehat{\mathcal{P}}_{s,\kappa}^{n}\!\left(r_{\z}(z)\right) \qquad \text{and}\qquad \mathcal{Q}_{s,\kappa}^n(z,\z):=\widehat{\mathcal{Q}}_{s,\kappa}^{n}\!\left(r_{\z}(z)\right),
\]
with
\begin{align}
	\widehat{\mathcal{P}}^n_{s, \kappa}(r):=&r^{|n|}\left(1-r^2\right)^s {}_2F_1\left(s-\sgnstr(n)\kappa, s+\sgnstr(n)\kappa+|n|; 1+|n|;r^2\right),\label{P}\\
	\widehat{\mathcal{Q}}^n_{s, \kappa}(r):=&-\frac{\Gamma(s-\sgnstr (n)\kappa)\Gamma(s+\sgnstr(n)\kappa +|n|)}{4\pi\Gamma(2s)}r^{-|n|}\left(1-r^2\right)^s\nonumber\\
	&\times{}_2F_1\left(s+\sgnstr(n)\kappa, s-\sgnstr(n)\kappa-|n|; 2s;1-r^2\right), \label{Q}
\end{align}
where for $n\in \Z$ we set $\sgnstr(n):=\sgn(n+1/2)$.

These functions are meromorphic in $s$ with at most simple poles at $s\in \pm \kappa -\N_0$ and satisfy certain useful relations.  Directly from the definitions \eqref{P} and \eqref{Q}, one obtains
\begin{align*}
\mathcal{P}_{s,\kappa}^{n}(z,\mathfrak{z})=\mathcal{P}_{s,-\kappa}^{-n}(z,\mathfrak{z}),\quad
\mathcal{Q}_{s,\kappa}^{n}(z,\mathfrak{z})=\mathcal{Q}_{s,-\kappa}^{-n}(z,\mathfrak{z}).
\end{align*}
Moreover, for $t\in\R$ we have
\begin{equation*}
\mathcal{P}_{t,\kappa}^{n}(z,\mathfrak{z}),\ \mathcal{Q}_{t,\kappa}^{n}(z,\mathfrak{z})\ \in\ \R.
\end{equation*}

The special values of $\mathcal{P}$ and $\mathcal{Q}$ in the cases $s=\kappa$ and $s=1-\kappa$ play an important role in our investigation.  To describe these, we set
\begin{equation}\label{eqn:cndef}
a_n=a_{\kappa,n}:= \frac{(-4)^{\kappa-1}}{\pi}
\begin{cases}
n! &\text{if } n\geq 0, \\
\frac{\Gamma (1-2\kappa -n)}{\Gamma(1-2\kappa)} &\text{if } n<0.
\end{cases}
\end{equation}
In the next section, we prove the following lemma.
\begin{lemma}\label{elreFay}
\begin{enumerate}[leftmargin=*,label={\rm(\arabic*)}]
\item For $\kappa\in-\N_0$, we have
	\begin{align}\label{Pid}
	y^{-\kappa}\left(\frac{z-\overline{\z}}{\z-\overline{z}}\right)^{-\kappa}\mathcal{P}^n_{1-\kappa, \kappa}(z, \z)e^{in\theta_{\z}(z)}
	&=\frac{\left(-4\y\right)^{\kappa}}{(z-\overline{\z})^{2\kappa}}X^n_\z(z) \begin{cases}
1 & \text{if }n\geq 0, \\
n \beta_0\left(1-r_{\z}^2(z);1-2\kappa,-n\right)
&\text{if }n< 0,
\end{cases}
\\ \label{Qid}
y^{-\kappa}\left(\frac{z-\overline{\z}}{\z-\overline{z}}\right)^{-\kappa}\mathcal{Q}^n_{1-\kappa, \kappa}(z, \z)e^{in\theta_{\z}(z)}
	&= \frac{a_n\y^{\kappa}}{(z-\overline{\z})^{2\kappa}}\beta\left(1-r^2_\z(z);1-2\kappa, -n\right)
X^n_\z(z).
\end{align}
\item
For $\kappa\in\N$ and $n\in-\N$, we have
\begin{equation}\label{eqn:Qs=kappa}
y^{-\kappa} \left(\frac{z-\overline{\z}}{\z-\overline{z}}\right)^{-\kappa}\mathcal{Q}_{\kappa, \kappa}^{n}(z, \z) e^{in\theta_{\z}(z)} = -\frac{(-n-1)!}{4\pi} \frac{(-4\y)^\kappa}{(z-\overline{\z})^{2\kappa}}X_\z^{n}(z).
\end{equation}
\item
For $\kappa \in \N$ and $n\in\N_0$, we have
\begin{equation}\label{eqn:Qs=kappapos}
\mathop{\lim}_{s \rightarrow \kappa} \frac{y^{-\kappa}}{\Gamma(s-\kappa)}  \left(\frac{z-\overline{\z}}{\z-\overline{z}}\right)^{-\kappa}\mathcal{Q}_{s, \kappa}^n(z, \z) e^{in\theta_{\z}(z)}= -\frac{(2\kappa+n-1)!}{4\pi (2\kappa-1)!} \frac{(-4\y)^\kappa}{
(z-\overline{\z})^{2\kappa}
}X_\z^n(z).
\end{equation}
\end{enumerate}
\end{lemma}
\indent
We next define certain Poincar\'{e} series considered by Fay. For this, we set (see \cite[(44)]{Fa}, slightly modified)
\begin{equation}\label{GFay}
\mathcal{G}^{m,n}_{s,\kappa}(z,\z):=c_{s,\kappa}^{m,n}y^{\kappa+n}\sum_{M\in\operatorname{SL}_2(\mathbb{Z})}\left. g_{s, \kappa}^{m,n}(z, \z)\right|_{2\kappa+2n, z}M
\end{equation}
with
\begin{equation}\label{eqn:cmndef}
c_{s,\kappa}^{m,n}:=(-1)^n
\begin{cases}
1 &\text{if }mn\geq0,\\
 \frac{\Gamma(s+\sgnstr(n)\kappa+\ell)\Gamma(s-\sgnstr(n)\kappa)}{\Gamma(s+\sgnstr(n)\kappa)\Gamma(s-\sgnstr(n)\kappa-\ell)}&\text{if }mn<0,\\
\end{cases}
\end{equation}
where $\ell:=\min\left(|m|,|n|\right)$ and
\begin{equation}\label{gFay}
g^{m,n}_{s,\kappa}(z, \z):=  y^{-\kappa-n}\left(\frac{\z-\overline{z}}{z-\overline{\z}}\right)^{\kappa+n}\mathcal{Q}^{-m-n}_{s, \kappa+n}(z,\z)e^{-i(m+n)\theta_{\mathfrak{z}}(z)}.
\end{equation}
\begin{remarks}
\noindent

\noindent
\begin{enumerate}[leftmargin=*,label={\rm(\arabic*)}]
\item
Note that if $c_{s,\kappa_2}^{m_2,n_2}=0$, then we multiply both sides of \eqref{GFay} by an appropriate factor to cancel the simple poles occurring in the $\Gamma$-factors and then take the limit, as in Lemma \ref{elreFay} (3).

\item
The functions $\mathcal{G}^{m,n}_{s,\kappa}$ satisfy the symmetry relations
\begin{align}
\label{eqn:Gshift}
\mathcal{G}^{m_1,n_1}_{s,\kappa_1}&
=\frac{c_{s,\kappa_1}^{m_1,n_1}}{c_{s,\kappa_2}^{m_2,n_2}}
 \mathcal{G}^{m_2,n_2}_{s,\kappa_2} \qquad\qquad\text{ if }\kappa_1+n_1=\kappa_2+n_2 \text{ and } m_1+n_1=m_2+n_2,\\
\label{eqn:Gconj}
\overline{\mathcal{G}^{m_1,n_1}_{s,\kappa_1}}&
=\frac{c_{\overline{s},\kappa_1}^{m_1,n_1}}{c_{\overline{s},\kappa_2}^{m_2,n_2}}
\mathcal{G}^{m_2,n_2}_{\overline{s},\kappa_2} \qquad\qquad\text{ if }\kappa_1+n_1=-\kappa_2-n_2 \text{ and } m_1+n_1=-m_2-n_2.
\end{align}
\end{enumerate}
\end{remarks}
Fay related these functions to the resolvent kernel $\mathcal{G}_{s,\kappa}:=\mathcal{G}_{s,\kappa}^{0,0}$.
\begin{theorem}{~\rm(Fay \cite[Theorem 2.1]{Fa})}\label{raise G}
For $\operatorname{Re}(s)>1, \mathfrak{z}\mapsto\mathcal{G}^{m,n}_{s,\kappa}(z, \z)\in \mathcal{F}_{-\kappa+m,s}$ and $z\mapsto\mathcal{G}^{m,n}_{s,\kappa}(z, \z)\in \mathcal{F}_{\kappa+n,s}$. We have, for $m,n\in\N_0$,
$$
\mathcal{G}^{m,n}_{s, \kappa}(z, \z)=
\mathcal{K}^m_{-\kappa, \z}
\circ\mathcal{K}^n_{\kappa, z}\left(\mathcal{G}_{s, \kappa}(z, \z)\right).
$$
If $m \text{ or }n<0$, then we replace $\mathcal{K}^j_{\ell, w}$ by $\mathcal{L}^{-j}_{\ell, w}$.
\end{theorem}

Fay also considered elliptic expansions of functions in $\mathcal{F}_{\kappa,s}$.
\begin{proposition}{~\rm(Fay \  \cite[Theorem 1.1]{Fa})}\label{Fay expand}
  If $\mathcal{D}_\kappa(g)=s(s-1)g$ in some annulus $A:r_1<d(z,\varrho)<r_2$ around $\varrho\in\mathbb{H}$, then $g$ has an elliptic expansion of the shape
\[
g(z)=\left(\frac{z-\overline{\varrho}}{\varrho-\overline{z}}\right)^{-\kappa}\sum_{n\in\mathbb{Z}}\left(c_{\varrho}(n)\mathcal{P}^n_{s, \kappa}(z,\varrho)+d_{\varrho}(n)\mathcal{Q}^n_{s, \kappa}(z,\varrho)\right)e^{in\theta_{\varrho}(z)}.
\]
\end{proposition}
The proof of Proposition \ref{prop:elliptic}, which we give in the next section, mostly relies on rewriting Fay's functions $\mathcal{P}_{s,\kappa}^n$ and $\mathcal{Q}_{s,\kappa}^n$.

\section{Special functions and elliptic expansions}\label{sec:ellexp}
To prove Proposition \ref{prop:elliptic} we write the elliptic expansion in terms of Fay's, which is done in Lemma \ref{elreFay} (1).
\begin{proof}[Proof of Lemma \ref{elreFay}]
(1) Throughout, we use the fact that, with $r:=r_{\z}(z)$, we have
\begin{equation}\label{eqn:y1-r^2}
y^{-\kappa}  \left(\frac{z-\overline{\z}}{\z-\overline{z}}\right)^{-\kappa} \left(1-r^2\right)^\kappa = (-4\y)^{\kappa}
(z-\overline{\z})^{-2\kappa}.
\end{equation}
For $n\geq0$, equation~\eqref{Pid} follows from the definition, using \eqref{Euler}, \eqref{eqn:2F10}, and \eqref{eqn:y1-r^2}.  If $n<0$ then, using \eqref{Euler}, \eqref{eqn:y1-r^2}, and
abbreviating $X:=X_{\z}(z)=r e^{i\theta}$ with $\theta:=\theta_{\z}(z)$, the left-hand side of \eqref{Pid} equals
\[
\left(-4\y\right)^{\kappa} X^n(z-\overline{\mathfrak{z}})^{-2\kappa} r^{-2n}  {}_2F_1\left(-n,2\kappa;1-n;r^2\right).
\]
Since $2\kappa\leq 0$ and $-n>0$, the claim follows from Lemma \ref{bet0lem}\ (2).

We next prove \eqref{Qid}. The claim for $n \geq 0$ follows from the definition using \eqref{Euler}, \eqref{bF}, and \eqref{eqn:y1-r^2}.  For $n<0$, the claim follows by \eqref{bF} and \eqref{eqn:y1-r^2}.

\noindent(2) From the definition of $\mathcal{Q}_{s, \kappa}^n(z, \z)$, the left-hand side of \eqref{eqn:Qs=kappa} equals
\[
-\frac{\Gamma(2\kappa)\Gamma(-n)}{4\pi \Gamma(2\kappa)} y^{-\kappa}\left(\frac{z-\overline{\z}}{\z-\overline{z}}\right)^{-\kappa}  r^n \left(1-r^2\right)^\kappa {}_2F_1\left(0, 2\kappa+n; 2\kappa; 1-r^2\right)e^{in\theta}.
\]
Once again using $re^{i\theta}=X$
and \eqref{eqn:2F10}, we obtain
\[
-\frac{(-n-1)!}{4\pi} y^{-\kappa}\left(\frac{z-\overline{\z}}{\z-\overline{z}}\right)^{-\kappa} X^n \left(1-r^2\right)^\kappa.
\]
By \eqref{eqn:y1-r^2}, we then obtain the claim.

\noindent
(3)
The left-hand side of \eqref{eqn:Qs=kappapos} equals
\[
-\mathop{\lim}_{s \rightarrow \kappa}  \frac{\Gamma(s-\kappa)\Gamma(s+\kappa+n)}{\Gamma(s-\kappa)4\pi \Gamma(2s)}
y^{-\kappa}\left(\frac{z-\overline{\z}}{\z-\overline{z}}\right)^{-\kappa} r^{-n} \left(1-r^2\right)^s {}_2F_1\left(s+\kappa, s-\kappa-n; 2s; 1-r^2\right)e^{in\theta}.
\]
Canceling $\Gamma(s-\kappa)$, using $X=re^{i\theta}$, taking the limit, employing \eqref{eqn:y1-r^2}, and plugging in the definition of the ${_2F_1}$, we obtain
\begin{equation}\label{eqn:elreFay3}
-\frac{\Gamma(2\kappa+n)\!\left(-4\y\right)^{\kappa}}{4\pi\Gamma(2\kappa)}
(z-\overline{\z})^{-2\kappa}
\frac{X^n}{r^{2n}}
{_2F_1}\!\left(2\kappa,-n;2\kappa;1-r^2\right).
\end{equation}
We obtain the desired identity by using 15.4.6 of \cite{NIST} to evaluate
\[
{_2F_1}\!\left(2\kappa,-n;2\kappa;1-r^2\right)=r^{2n}.
\]
\end{proof}

We next combine Lemma \ref{elreFay} (1) with Fay's elliptic expansion in Proposition \ref{Fay expand} to obtain Proposition \ref{prop:elliptic}.

\begin{proof}[Proof of Proposition \ref{prop:elliptic}]

With $G(z):=y^\kappa F(z)$, we have, by \eqref{DDel},
$$
\mathcal{D}_\kappa\left(G(z)\right)=-y^\kappa \Delta_{2\kappa}\left(F(z)\right)-\kappa(1-\kappa)G(z)=\kappa(\kappa-1)G(z)
$$
since $\Delta_{2\kappa}(F)=0$ by assumption.  Thus, by Proposition \ref{Fay expand} with $s=1-\kappa$,
$$
G(z)=\left(\frac{z-\overline{\varrho}}{\varrho-\overline{z}}\right)^{-\kappa}\sum_{n\in\mathbb{Z}} \left(c_\varrho(n)) \mathcal{P}^n_{1-\kappa, \kappa}(z, \varrho)+d_\varrho(n)\mathcal{Q}^n_{1-\kappa, \kappa}(z,\varrho)\right)e^{in
\theta_{\varrho}(z)}.
$$
By Lemma \ref{elreFay} (1) and \eqref{eqn:beta0def}, this gives
\[
F(z)=
\left(z-\overline{\varrho}\right)^{-2\kappa}
\sum_{n\in\mathbb{Z}}\alpha_{\varrho}(n)
X_{\varrho}^n(z)+\left(z-\overline{\varrho}\right)^{-2\kappa}
\sum_{n\in\mathbb{Z}}\gamma_{\varrho}(n)\beta_0\left(1-
r_{\varrho}^2(z);
1-2\kappa, -n\right)
X_{\varrho}^n(z).
\]
for some constants $\alpha_{\varrho}(n),\gamma_{\varrho}(n)\in\C$. Rewriting yields the expansion \eqref{eqn:elliptic} up to the restrictions on $n$ in each of the sums.
It thus remains to show that $c_{F,\varrho}^+(n)=0$ for $n<-n_0$ and $c_{F,\varrho}^-(n)=0$ for $n>n_0$.  To do so, we investigate the asymptotic growth of each term in the sum as $z\to \varrho$.  We repeatedly use the fact that, as $z\to \varrho$, $X_{\varrho}(z)\sim_{\eta}z-\varrho$, where by $G_1(z,\varrho)\sim_{\eta}G_2(z,\varrho)$ we mean that there is a constant $C_{\eta}\neq 0$ depending only on $\eta$ such that $\lim_{z\mapsto \varrho}\frac{G_1(z,\varrho)}{C_{\eta}G_2(z,\varrho)}=1$.  This gives that $n\geq -n_0$ for the first summand in \eqref{eqn:elliptic}.

Moreover, by Lemma \ref{bet0lem},
\[
\beta_0\left(1-r^2;1-2\kappa, -n\right)=\sum_{\substack{0\leq j\leq -2\kappa\\ j\neq n}} \binom{-2\kappa}{\hspace*{1.5mm}j}\frac{(-1)^{j+1}}{j-n}r^{2(j-n)}+2\delta_{0\leq n\leq-2\kappa}\binom{-2\kappa}{\hspace*{1.5mm}n}(-1)^{n+1}\log\left(r\right).
\]
Thus, again using $X_{\varrho}^n(z)\sim_{\eta} (z-\varrho)^n$, we have, as $z\to \varrho$,
\begin{align}
\nonumber	\beta_0\left(1-r_{\varrho}^2(z); 1-2\kappa, -n\right)X_{\varrho}^n(z) &\sim_{\eta} r_{\varrho}^{-2n}(z)X_{\varrho}^n(z)+\delta_{0\leq n\leq-2\kappa}\log\left(r_{\varrho}(z)\right)X_{\varrho}^n(z)\\
	&\sim_{\eta} (z-\varrho)^{-n}+\delta_{0\leq n\leq-2\kappa}(z-\varrho)^n\operatorname{Log}(z-\varrho).\label{eqn:beta0asymptotic}
\end{align}
Furthermore, for $0<n<n_{\kappa}$, since $X_{\varrho}^n(z)\to 0$ as $z\to \varrho$, the asymptotic in \eqref{eqn:beta0asymptotic} implies that we also have
\[
\beta\left(1-r_{\varrho}^2(z);1-2\kappa,-n\right)X_{\varrho}^n(z)=\beta_0(1-r_{\varrho}^2(z);1-2\kappa,-n)X_{\varrho}^n(z)+C_{1-2\kappa,-n}X_{\varrho}^n(z)\sim_{\eta}(z-\varrho)^{-n}.
\]
This gives the claimed bounds for $n\neq 0$.  Finally the $n=0$ term behaves like $\operatorname{Log} (z-\varrho)$ by \eqref{eqn:beta0asymptotic}. This growth is cancelled upon multiplying by $r_{\varrho}^{n_0}(z)$.

\noindent
(2) By  \eqref{eqn:rdef} and \eqref{eqn:dinv},
for $M$ in the stabilizer $\Gamma_{\varrho}\subset \SL_2(\Z)$ of $M$
\begin{equation}\label{eqn:rinv}
r_{\varrho}(Mz)=r_{M\varrho}(Mz)=r_{\varrho}(z).
\end{equation}
One concludes the claim by \cite[(2a.16)]{Pe1}.

\end{proof}

We next compute the action of differential operators on elliptic expansions  in Proposition \ref{prop:elliptic}.
\begin{proof}[Proof of Proposition \ref{Explicit}]

We first note that, by Proposition \ref{prop:LRKL} (1),
\begin{equation}\label{DK}
D^{2k-1}\left(F(z)\right)=(-4\pi)^{1-2k}y^{-k}\mathcal{K}_{1-k}^{2k-1}\left(G(z)\right).
\end{equation}
We rewrite the right-hand side of \eqref{DK} in terms of the iterated operators (for $\ell\in\N_0$
and $\kappa\in\Z$)
\begin{align*}
\widehat{\mathcal{K}}_{
\kappa,
n}^{\ell}:= \widehat{\mathcal{K}}_{
\kappa
+\ell-1,n+1-\ell}\circ\cdots\circ\widehat{\mathcal{K}}_{
\kappa,
n},\quad
\widehat{\mathcal{L}}_{
\kappa,
n}^{\ell}:= \widehat{\mathcal{L}}_{
\kappa
+1-\ell,n+\ell-1}\circ\cdots\circ\widehat{\mathcal{L}}_{
\kappa,
n},
\end{align*}
where (see \cite[after formula (18)]{Fa})
\begin{align*}
\widehat{\mathcal{K}}_{\kappa,n}:= \frac{1}{2}\left(1-r^2\right)\frac{\partial}{\partial r} + \left( \frac{2nr}{1-r^2} - \kappa r\right),\quad
\widehat{\mathcal{L}}_{\kappa,n}:= \frac{1}{2}\left(1-r^2\right)\frac{\partial}{\partial r} - \left( \frac{2nr}{1-r^2} - \kappa r\right).
\end{align*}
Namely, using (see \cite[(14)]{Fa}) that for $f: \R_0^+\to\C$
\begin{multline*}
 \mathcal{K}_{\kappa}\left(\left(\frac{z-\overline{\varrho}}{\varrho-\overline{z}}\right)^{-\kappa}f\!\left(
r_{\varrho}(z)
\right)e^{in
\theta_{\varrho}(z)
}\right)\\
=\left(\frac{z-\overline{\varrho}}{\varrho-\overline{z}}\right)^{-\kappa-1}e^{-i
\theta_{\varrho}(z)
}	\left[\left(\frac{1}{2}\left(1 - r^2\right)\frac{\partial}{\partial r} -\kappa r- \frac{2 i r}{1-r^2}\frac{\partial}{\partial \theta}\right)e^{in\theta}f(r)\right]_{
\substack{r=r_{\varrho}(z),\\ \theta=\theta_{\varrho}(z)}
}
\end{multline*}
and iteratively carrying out the differentiation on $\theta$ yields
\begin{equation}\label{Kl}
\mathcal{K}_\kappa^{\ell}\left(\left(\frac{z-\overline{\varrho}}{\varrho-\overline{z}}\right)^{-\kappa}f\!\left(
r_{\varrho}(z)
\right)e^{in
\theta_{\varrho}(z)
}\right)=\left(\frac{z-\overline{\varrho}}{
\varrho
-\overline{z}}\right)^{-\kappa-\ell}e^{i(n-\ell)
\theta_{\varrho}(z)
}\left[\widehat{\mathcal{K}}_{\kappa,n}^{\ell}\left(f(r)\right)\right]_{
r=r_{\varrho}(z)
}.
\end{equation}

\hspace{-10pt}By \eqref{DK} and \eqref{Pid}, we thus have, for $n\geq 0$,
\begin{align}
\nonumber &D_z^{2k-1}\left(
\left(z-\overline{\varrho}\right)^{2k-2}
X_{\varrho}^{n}(z)\right)=(-4\pi)^{1-2k}y^{-k}\left(-4\eta\right)^{k-1}\mathcal{K}_{1-k,z}^{2k-1}\left(\left(\frac{z-\overline{\varrho}}{\varrho-\overline{z}}\right)^{k-1}\mathcal{P}_{k,1-k}^n(z,\varrho)e^{in\theta_{\varrho}(z)
}\right)\\
\label{eqn:DXdisc} &=(-4\pi)^{1-2k}y^{-k}\left(\frac{z-\overline{\varrho}}{\varrho-\overline{z}}\right)^{-k}e^{i(n+1-2k)
\theta_{\varrho}(z)
}\left(-4\eta\right)^{k-1}\left[\widehat{\mathcal{K}}_{1-k, n}^{2k-1}\left(\widehat{\mathcal{P}}^n_{k,1-k}(r)\right)\right]_{r=r_{\varrho}(z)}.
\end{align}

By \cite[(18)]{Fa}, we know that
\begin{align}
\label{eqn:KP}
\widehat{\mathcal{K}}_{\kappa, n}\left(\widehat{\mathcal{P}}_{s, \kappa}^n(r)\right)
&=
e_{s,\kappa}(n)\widehat{\mathcal{P}}_{s, \kappa+1}^{n-1}(r),\\
\label{eqn:KQ}\widehat{\mathcal{K}}_{\kappa, n}\left(\widehat{\mathcal{Q}}_{s, \kappa}^n(r)\right)&=d_{s,\kappa}(n)\widehat{\mathcal{Q}}_{s, \kappa+1}^{n-1}(r),
\end{align}
where
\begin{align*}
e_{s,\kappa}(n)
&:=
\begin{cases}
n&\qquad \text{ if } n\geq 1,\\
\frac{(s+\kappa)(s-\kappa-1)}{1+|n|}&\qquad \text{ if } n\leq 0,
\end{cases}
\qquad\quad
d_{s,\kappa}(n)
&:=
\begin{cases}
-(s+\kappa)(s-\kappa-1)&\qquad \text{ if } n\geq 1,\\
-1&\qquad \text{ if } n\leq 0.
\end{cases}
\end{align*}

Plugging \eqref{eqn:KP} into the right-hand side of \eqref{eqn:DXdisc}
simplifies to
\begin{equation}\label{eqn:DX2}
\frac{\prod_{j=1}^{2k-1}e_{k,j-k}(n+1-j)}{(-4\pi)^{2k-1}(-4\eta)^{1-k}}y^{-k}\left(\frac{z-\overline{\varrho}}{\varrho-\overline{z}}\right)^{\kappa-1}e^{i(n+1-2k)
\theta_{\varrho}(z)
} \mathcal{P}_{k,k}^{n+1-2k}(z,\varrho).
\end{equation}
We split into the cases $n\geq 2k-1$ and $
n<2k-1$.

For $n\geq 2k-1$ we have, using \eqref{eqn:2F10},
\[
\mathcal{P}_{k,k}^{n+1-2k}(z,
\varrho) =
r_{\varrho}^{n+1-2k}(z)
\!\left(1-
r_{\varrho}^{2}(z)
\right)^{k}{_2F_1}\!\left(0,1+n; n+2-2k;
r_{\varrho}^2(z)
\right)=
r_{\varrho}^{n+1-2k}(z)
\!\left(1-
r_{\varrho}^{2}(z)
\right)^{k}.
\]
Thus \eqref{eqn:DX2} becomes, using that $
r_{\varrho}(z)e^{i\theta_\varrho(z)}=X_{\varrho}(z)
$ and \eqref{eqn:y1-r^2},
\[
\pi^{1-2k}\eta^{2k-1}\prod_{j=1}^{2k-1}e_{k,j-k}(n+1-j)(z-\overline{\varrho})^{-2k}X_{\varrho}^{n+1-2k}(z).
\]
Explicitly computing the constant then finishes the claim for $n\geq 2k-1$.
For $0\leq n\leq 2k-2$,
we have $e_{k,k-1}(n+2k-2)= 0$, giving the claim in this range.

We next act by $D_z^{2k-1}$ on the non-meromorphic part of $F$.  First assume that $n\notin [0,2k-2]$. By Lemma \ref{bet0lem} (1), we then have,
using that $r_{\varrho}^2(z)=X_{\varrho}(z)\overline{X_{\varrho}(z)}$,
\[
(z-\overline{\varrho})^{2k-2} \beta_0\!\left(1-r_{\varrho}^2(z); 2k-1,-n\right)X_{\varrho}^n(z)=\sum_{0\leq j\leq 2k-2} \binom{2k-2}{j}\frac{(-1)^{j+1}}{j-n} (z-\overline{\varrho})^{2k-2}X_{\varrho}^{j}(z) \overline{X_{\varrho}^{j-n}(z)}.
\]
This is a polynomial in $z$ of degree at most $2k-2$ (with antiholomorphic coefficients depending on $\overline{z}$).  Differentiating $2k-1$ times hence annihilates these terms.

It remains to determine the image of $D^{2k-1}$ on the terms in the non-meromorphic part with $0\leq n\leq 2k-2$.
Using \eqref{Qid}, \eqref{DK}, \eqref{Kl}, \eqref{eqn:KQ}, and \eqref{eqn:Qs=kappa}, we obtain that
\begin{equation*}
\begin{split}
&D_z^{2k-1}\left((z-\overline{\varrho})^{2k-2} \beta\!\left(1-r_{\varrho}^2(z);2k-1,-n\right)X_{\varrho}^n(z)
\right)\\
 &\qquad=\frac{(2k-2-n)!\prod_{j=1}^{2k-1}d_{k,j-k}(n+1-j)}{a_n(-4\pi)^{2k}
\eta
^{1-2k}(-4)^{-k}}(z-\overline{\varrho})^{-2k} X_\varrho^{n+1-2k}(z).
\end{split}
\end{equation*}
Computing
\[
\prod_{j=1}^{2k-1}d_{k,j-k}(n+1-j)=
-\frac{n!(2k-2)!}{(2k-2-n)!}
\]

and plugging in \eqref{eqn:cndef}, noting that $n\geq 0$,  yields the claimed formula.
\end{proof}

\section{Poincar\'{e} series and the proof of Theorem \ref{OperatorTheorem}}\label{sec:Poincare}
For $\z\in\mathbb{H}$, $n\in\mathbb{Z}$, and $k > 1$, we define the meromorphic Poincar\'{e} series, due to Petersson,
\begin{equation}\label{Psi}
\Psi_{2k, n}^{\z}(z)=\Psi_{2k, n}(z, \z):=\sum_{M\in\operatorname{SL}_2(\mathbb{Z})}\left. \psi_{2k,n}^{\z}(z)\right|_{2k}M,
\end{equation}
where
\[
\psi_{2k,n}^{\z}(z)=\psi_{2k,n}(z,\z):=(z-\overline{\mathfrak{z}})^{-2k}X_{\mathfrak{z}}^n(z).
\]
We use the convention that $\z$ appears as a superscript in the notation if we consider $\z\in\H$ as a fixed point and we write it as a two-variable function if we consider the properties in the $\z$-variable.
The main properties of $\Psi_{2k,n}^{\z}$ needed for this paper are given in the proposition below.
\noindent
\begin{proposition}{~\rm(Petersson \cite[S\"atze 7 and 8]{Pe2} and \cite[Satz 7]{PeEinheit})}\label{Psiprop}
The functions  $\{\Psi_{2k, n}^{\z}: \z\in \H, n\in \Z\}\ ($resp. $\{\Psi_{2k, n}^{\z}: n\in \Z\})$ span $\mathbb{S}_{2k}\ ($resp. $\mathbb{S}_{2k}^{\z})$.
 For $n\geq0$ they are cusp forms. For $n<0$ they are orthogonal to cusp forms and have the principal part $2\omega_{\z} \psi_{2k,n}(z,\z)$ around $z=\z$.
\end{proposition}
\begin{remarks}
\noindent
\noindent
\begin{enumerate}[leftmargin=*,label={\rm(\arabic*)}]
\item
By Proposition \ref{Psiprop}, the elliptic expansion of $\Psi_{2k,n}^{\z
}$ around $\varrho\in\H$ may be written
\begin{equation}\label{eqn:Psiexpand}
\Psi_{2k,n}^{\z}(z)=\left(z-\overline{\varrho}\right)^{-2k}\left(2\omega_{\z}\delta_{[\varrho]=[\z]} \delta_{n<0}X_{\varrho}^n(z)+ \sum_{\ell\geq 0}c_{2k,\varrho}^{\z}(n,\ell)X_{\varrho}^{\ell}(z)\right),
\end{equation}
where $[\z]$ denotes the $\SL_2(\Z)$-equivalence class of $\z$.
\item
As pointed out in \cite[page 72]{Pe1}, $\z\mapsto \y^{2k+n}\Psi_{2k,n}^{\z}(z)$ is modular of weight $-2k-2n$.  Moreover, it is an eigenfunction under $\Delta_{-2k-2n,\z}$ with eigenvalue $(2k+n)(n+1)$.
\end{enumerate}
\end{remarks}
We next write $\Psi$ as a special case of Fay's function $\mathcal{G}$.
We set
\begin{equation*}
\mathscr{C}_{k,n} :=\begin{cases}
n!(2k-2)! & \textrm{if } n\geq 0,\\
(2k-2-n)! & \textrm{if } n < 0.
\end{cases}
\end{equation*}

\begin{lemma}\label{lem:PsiGrel}
\noindent

\noindent
\begin{enumerate}[leftmargin=*,label={\rm(\arabic*)}]
\item
We have
\begin{equation*}
\lim_{s\to k}\mathcal{G}_{s, 1-k}^{-n, 2k-1}(z, \z)=   \frac{\left(-4\y y\right)^k}{4\pi}
\mathscr{C}_{k,n} \Psi_{2k, n+1-2k}^{\z}(z).
\end{equation*}
\item
If $n\in\N_0$, then
\[
\mathcal{G}_{k,k-1}^{n,1} (z,\z)=\frac{(-4\y y)^{k}}{4\pi} n! \Psi_{2k,-n-1}^{\z}(z).
\]
If $n\in-\N
$, then
\[
\lim_{s\to k} \mathcal{G}_{s,k-1}^{n,1} (z,\z)=\frac{(-4\y y)^{k}}{4\pi} \frac{(2k-2-n)!}{(2k-2)!}\Psi_{2k,-n-1}^{\z}(z).
\]
\end{enumerate}
\end{lemma}
\begin{proof}
(1) By inspecting the definitions \eqref{GFay} and \eqref{Psi} the claim follows once we show that
\begin{equation}\label{eqn:gXrel}
\lim_{s\to k}\left( c_{s, 1-k}^{-n, 2k-1} y^k g_{s, 1-k}^{-n, 2k-1}(z, \z)\right) = \frac{(-4\y y)^k}{4\pi}
(z-\overline{\z})^{-2k}
\mathscr{C}_{k,n} X^{n-2k+1}.
\end{equation}
By definition \eqref{gFay}, we have
\[
g_{s, 1-k}^{-n, 2k-1}(z, \z) = y^{-k} \left(\frac{\z-\overline{z}}{z-\overline{\mathfrak{z}}}\right)^{k} \mathcal{Q}_{s, k}^{n-2k+1}(z, \z)e^{i(n-2k+1)\theta}.
\]
If $n <2k-1$, then we may plug in $s=k$ directly and then use Lemma~\ref{elreFay} (2) to obtain
\[
 \lim_{s\to k} g_{s,1-k}^{-n,2k-1}(z,\z) =-\frac{(2k-2-n)!}{4\pi} (-4\y)^k
(z-\overline{\z})^{-2k}  X^{n-2k+1}.
\]
For $n\leq 0$, we have $c_{k,1-k}^{-n,2k-1} =(-1)$ implying \eqref{eqn:gXrel} in this case.\\
\indent For $0<n<2k-1$, we obtain \eqref{eqn:gXrel} for $n< 2k-1$, computing
\[
c_{k,1-k}^{-n,2k-1} = -\frac{n!(2k-2)!}{(2k-n-2)!}.
\]

If $n\geq 2k-1$, then $\ell=2k-1$ in \eqref{eqn:cmndef} and we use Lemma~\ref{elreFay} (3) (replacing $n$ by $n-2k+1$) to obtain the desired formula.

(2)  By \eqref{eqn:Gshift}, we have
\begin{equation}\label{shiftG}
\mathcal{G}_{s,k-1}^{n,1}(z,\z) = \frac{c_{s,k-1}^{n,1}}{c_{s,1-k}^{n+2-2k,2k-1}}\mathcal{G}_{s,1-k}^{n+2-2k, 2k-1}(z,\z).
\end{equation}
For $n\geq 0$, we may then directly plug in $s=k$ and use (1) to obtain the claim by simplifying, with $n\mapsto k-2-n$,
\[
\frac{c_{k,k-1}^{n,1}}{c_{k,1-k}^{n+2-2k,2k-1}} \mathscr{C}_{k,2k-2-n}=n!.
\]

For $n<0$, we use (1) to obtain, by \eqref{shiftG},
\begin{align*}
\lim_{s\to k}\mathcal{G}_{s,k-1}^{n,1}(z,\z) &= \frac{\lim_{s\to k}\left(\Gamma(s-k)c_{s,k-1}^{n,1}\right)}{\lim_{s\to k}\left(\Gamma(s-k)c_{s,1-k}^{n+2-2k,2k-1}\right)}\lim_{s\to k}\mathcal{G}_{s,1-k}^{n+2-2k, 2k-1}(z,\z)\\
&=\frac{\lim_{s\to k}\left(\Gamma(s-k) c_{s,k-1}^{n,1}\right)}{\lim_{s\to k}\left(\Gamma(s-k)c_{s,1-k}^{n+2-2k,2k-1}\right)} \frac{(-4\y y)^k}{4\pi}
\mathscr{C}_{k,2k-2-n}
 \Psi_{2k,-n-1}^{\z}(z).
\end{align*}
We then plug in \eqref{eqn:cmndef} and take the limit to obtain the claim.
\end{proof}

Next define for $n\in\Z$ the following polar harmonic Maass Poincar\'e series
\begin{equation}\label{eqn:Pdef}
\mathbb{P}_{2-2k,n}^{\z}(z)=\mathbb{P}_{2-2k,n}(z, \z):=\sum_{M\in\operatorname{SL}_2(\mathbb{Z})}\left.\varphi_{2-2k,n}^{\z}(z)\right|_{2-2k, z}M,
\end{equation}
with, using \eqref{Qid},
\begin{align}\label{eqn:varphidef}
\varphi_{2-2k,n}^{\z}(z)
&=\varphi_{2-2k,n}(z,\z):=
(z-\overline{\z})^{2k-2}
\beta\left(1-r^2_\z(z);2k-1, -n\right)X^n_\z(z)\\
&=\frac{\y^{k-1}y^{k-1}}{a_{1-k,n}}\left(\frac{z-\overline{\z}}{\z-\overline{z}}\right)^{k-1}\mathcal{Q}_{k,1-k}^{n}(z,\z)e^{in\theta_{\z}(z)}. \nonumber
\end{align}

The following more precise version of Theorem \ref{OperatorTheorem} shows how the functions $\mathbb{P}_{2-2k,n}^{\z}$ are related to the functions $\Psi_{2k,n}^{\z}$ via differential operators.
\begin{theorem}\label{thm:PoincProperties}
Assume $k\in \mathbb{N}_{>1}$.  The functions $\{\mathbb{P}_{2-2k, n}^{\z}: \z\in\H,\ n\in\Z\}$ {\rm(}resp. $\{\mathbb{P}_{2-2k, n}^{\z}: n\in\Z\}${\rm)} span the space $\mathscr{H}_{2-2k}$ {\rm(}resp. $\mathscr{H}_{2-2k}^{\z}${\rm)}.   Moreover
\begin{align}
\label{eqn:DPoincProperties}  D^{2k-1}\left(\mathbb{P}_{2-2k,n}^{\z}\right)&=-(2k-2)!\left(\frac{\y}{\pi}\right)^{2k-1}\Psi_{2k, n+1-2k}^{\z},\\
\label{eqn:xiPoincProperties}
\xi_{2-2k}\left(\mathbb{P}_{2-2k,n}^{\z}\right)&=(4\y)^{2k-1}\Psi_{2k, -n-1}^{\z}.
\end{align}
The functions $\mathbb{P}_{2-2k,n}^{\z}$ vanish unless $n\equiv k-1\pmod {\omega_\z}$, in which case their principal parts equal
$$
2\omega_{\z}
(z-\overline{\z})^{2k-2}
 X^n_\z(z)
\begin{cases}
	\beta_0\left(1-r^2_\z(z);2k-1, -n\right) & \text{if }n>2k-2,\\
	\beta\left(1-r^2_\z(z);2k-1, -n\right) & \text{if }0\leq n\leq 2k-2,\\
\mathcal{C}_{2k-1,-n}
&\text{if }n<0.
\end{cases}
$$
\end{theorem}
\begin{remarks}

\noindent

\noindent
\begin{enumerate}[leftmargin=*,label={\rm(\arabic*)}]
\item
It is also natural to ask about the properties of $\z\mapsto \mathbb{P}_{2-2k,n}^{\z}(z)$ for fixed $z$.  To investigate this, note that by a comparison of definitions \eqref{eqn:Pdef} and \eqref{GFay}, we find that
\begin{equation}\label{eqn:PG}
\mathbb{P}_{2-2k, n}^{\z}(z) = \frac{(\y y)^{k-1}}{a_{1-k, n}} \mathcal{G}_{k, 1-k}^{-n, 0}(z, \z),
\end{equation}
with $a_{1-k,n}$ given in \eqref{eqn:cndef}, and we evaluate $c_{k,1-k}^{-n,0}=1$ via \eqref{eqn:cmndef}.
Combining this with Theorem \ref{raise G}, one can conclude that the function $\z\mapsto \y^{n+2-2k}\mathbb{P}_{2-2k, n}^{\z}(z)$ has weight $2k-2-2n$ and eigenvalue $(n+1)(n+2-2k)$ under $\Delta_{2k-2-2n,\mathfrak z}$.

\item
By Theorem \ref{thm:PoincProperties} and Proposition \ref{prop:elliptic}, one may write the elliptic expansion of $\mathbb{P}_{2-2k,n}^{\z}$ around $\varrho$ for $n<0$ and $\z\in \H$ as
\begin{multline}\label{eqn:Pexpand}
\mathbb{P}_{2-2k,n}^{\z}(z) =\mathcal{C}_{2k-1,-n}\left(z-\overline{\varrho}\right)^{2k-2}\bigg(2\omega_{\z} \delta_{[\varrho]=[\z]}X_{\varrho}^{n}(z) + \sum_{\ell\geq 0}c_{2-2k,\varrho}^{\z,+}(n,\ell) X_{\varrho}^{\ell}(z)\\
+ \sum_{\ell<0}c_{2-2k,\varrho}^{\z,-}(n,\ell) \beta_0\!\left(1-r_{\varrho}^2(z); 2k-1,-\ell\right) X_{\varrho}^{\ell}(z)\bigg).
\end{multline}
\end{enumerate}
Furthermore, using \eqref{eqn:Cdef}, Lemma \ref{bet0lem} (1), and then Lemma \ref{bet0lem} (2) yields
\begin{equation}\label{eqn:n<0singconst}
\mathcal{C}_{2k-1,-n}= \beta\!\left(1;-n,2k-1\right)= \beta(-n,2k-1)=\frac{(-n-1)!(2k-2)!}{(2k-2-n)!},
\end{equation}
giving the constant in \eqref{eqn:n<0sing} in terms of factorials.
\end{remarks}

\begin{proof}[Proof of Theorem \ref{thm:PoincProperties}]

Using \eqref{GFay}, we conclude that $\mathbb{P}_{2-2k,n}^{\z}$ satisfies weight $2-2k$ modularity.

The principal part of $\mathbb{P}_{2-2k,n}^{\z}$ around $\z$ comes from the terms $M\in\Gamma_{\z}$ in \eqref{eqn:Pdef}.
Vanishing of the principal part for $n\not\equiv k-1\pmod{\omega_{\z}}$ follows from \eqref{eqn:rinv} together with \cite[(2a.16)]{Pe1}. For $n\equiv k-1\pmod{\omega_{\z}}$, this yields $2\omega_{\z}\varphi_{2-2k,n}^{\z}$.
For $0\leq n\leq 2k-2$, this directly yields the principal part.  For $n>2k-2$ or $n<0$, we rewrite the incomplete beta function using \eqref{eqn:beta0def} and note that
for $n>2k-2$ only the non-meromorphic part grows as $z$ approaches $\z$,
while for $n<0$ only the meromorphic part grows.

Since every possible principal part in the elliptic expansion of an element of $\mathscr{H}_{2-2k}^{\z}$ is obtained as a linear combination of the Poincar\'e series $\mathbb{P}_{2-2k,n}^{\z} ( n\in\Z)$, these span the space $\mathscr{H}_{2-2k}^{\z}$.  Moreover, eliminating the principal parts at different points in $\H$ one at a time implies that the space $\mathscr{H}_{2-2k}$ is spanned by $\{\mathbb{P}_{2-2k,n}^{\z}: \z\in\H,\ n\in\Z \}$.

We next compute the image of the Poincar\'e series under $D^{2k-1}$.  Using \eqref{eqn:PG} and \eqref{eqn:KRrep}, we obtain
\begin{equation}\label{eqn:raiseP1}
R_{2-2k,z}^{2k-1}\left(\mathbb{P}_{2-2k,n}^{\z}(z)\right) = \frac{\y^{k-1} y^{-k}}{a_{1-k, n}}\mathcal{K}_{1-k,z}^{2k-1}\left(\mathcal{G}_{k, 1-k}^{-n, 0}(z, \z)\right).
\end{equation}
Using Theorem \ref{raise G} twice, we then find that
\[
\mathcal{K}_{1-k,z}^{2k-1}\left(\mathcal{G}_{k, 1-k}^{-n, 0}(z, \z)\right) = \lim_{s\to k}\mathcal{G}_{s,1-k}^{-n,2k-1}(z,\z).
\]
We next employ Lemma \ref{lem:PsiGrel} (1) and plug back into \eqref{eqn:raiseP1}, yielding
\[
R_{2-2k}^{2k-1}\left(\mathbb{P}_{2-2k,n}^{\z}(z)\right) =\frac{(-4)^k\y^{2k-1}\mathscr{C}_{k,n} }{4\pi a_{1-k, n} }\Psi_{2k,n+1-2k}^{\z}(z).
\]
We then plug in the definitions of $\mathcal C_{k,n}$ and $a_{1-k,n}$ and use \eqref{Bol}
to conclude \eqref{eqn:DPoincProperties}.

It remains to compute the image under $\xi_{2-2k}$.  Firstly, by Proposition \ref{prop:LRKL} (1), with $f:\mathbb H\to \C$, we have
$$
\xi_{2-2k}\left(f(z)\right)=y^{-2k}\overline{L\left(f(z)\right)}=y^{-k}\overline{\mathcal{L}_{1-k}\left(y^{1-k}f(z)\right)}.
$$
Using \eqref{eqn:PG} and \eqref{eqn:cmndef} thus gives
\begin{equation*}
	\xi_{2-2k}\left(\mathbb{P}_{2-2k,n}^\z(z)\right)=\frac{\y^{k-1}}{a_{1-k,n}} y^{-k}\overline{\mathcal{L}_{1-k,z}\left(\mathcal{G}_{k,1-k}^{-n,0}(z,\z)\right)}.
	\end{equation*}

\noindent Now by Theorem \ref{raise G}, we have
\begin{equation*}
\mathcal{G}_{k,1-k}^{-n,0}(z,\z)=\begin{cases}
	\mathcal{L}_{k-1,\z}^n\left(\mathcal{G}_{k,1-k}(z,\z)\right)\qquad \text{if } n\geq 0,\\
	\mathcal{K}_{k-1,\z}^{-n}\left(\mathcal{G}_{k,1-k}(z, \z)\right)\qquad \ \text{if } n<0.
	\end{cases}
\end{equation*}
Again using Theorem \ref{raise G} and then applying \eqref{eqn:Gconj} gives
\begin{equation*}
\overline{\mathcal{L}_{1-k,z}\left(\mathcal{G}_{k,1-k}^{-n,0}(z,\z)\right)}=\lim_{s\to k}\overline{\mathcal{G}_{s,1-k}^{-n,-1}(z,\z)}=\lim_{s\to k}\mathcal{G}_{s,k-1}^{n,1}(z,\z).
\end{equation*}
For $n\geq 0$, we then use Lemma \ref{lem:PsiGrel} (2) to obtain
\[
\xi_{2-2k}\left(\mathbb{P}_{2-2k,n}^\z(z)\right)=\frac{\y^{k-1}}{a_{1-k,n}} \frac{(-4\y)^{k}}{4\pi} n! \Psi_{2k,-n-1}^{\z}(z).
\]
We then simplify the factor in front using \eqref{eqn:cndef} to obtain the claim for $n\geq 0$.
For $n<0$, we use Lemma \ref{lem:PsiGrel} (2) to obtain that
\[
\xi_{2-2k}\left(\mathbb{P}_{2-2k,n}^\z(z)\right)=\frac{\y^{k-1}}{a_{1-k,n}} \frac{(-4\y)^{k}}{4\pi} \frac{(2k-2-n)!}{(2k-2)!}\Psi_{2k,-n-1}^{\z}(z).
\]
Simplifying the constant yields the claim.

\end{proof}

We may now combine the results in this section to obtain Theorem \ref{OperatorTheorem}.
\begin{proof}[Proof of Theorem \ref{OperatorTheorem}]

(1)  Part (1) is the first statement in  Theorem \ref{thm:PoincProperties}.

\noindent
(2)  The claim follows from \eqref{eqn:xiPoincProperties} and \eqref{eqn:DPoincProperties} together with the fact that $\Psi_{2k,n}^{\z}\in S_{2k}$ if and only if $n\geq 0$ and $\Psi_{2k,n}^{\z}\in \mathbb{D}_{2k}^{\z}$ if and only if $n\leq -2k$.  These claims about $\Psi_{2k,n}^{\z}$ follow in turn from the principal parts and orthogonality given in Proposition \ref{Psiprop}.

\noindent

\noindent
(3) This follows by \eqref{eqn:xiPoincProperties}, \eqref{eqn:DPoincProperties}, and Proposition \ref{Psiprop}, since $\Psi_{2k,n}^{\z}\in \mathbb{E}_{2k}^{\z}$ if and only if $-2k<n<0$.

\noindent
(4)  The statements given here are precisely \eqref{eqn:xiPoincProperties} and \eqref{eqn:DPoincProperties}.

\end{proof}

\section{Duality, orthogonality, and the proof of Theorem \ref{DualityTheorem}}\label{sec:duality}

\subsection{Definition of the inner product}\label{sec:innerdef}
Petersson defined a regularized inner product (see \cite[p. 34]{Pe2}) for meromorphic modular forms by taking the Cauchy principal value of the naive definition.  More precisely, suppose that all of the poles of $f,g\in \mathbb{S}_{2k}$ in $\SL_2(\Z)\backslash \H$ are at the points $\z_1,\dots,\z_r$,
where we abuse notation to allow $\z_{\ell}$ to denote both the coset $[\z_{\ell}]$ is $\SL_2(\Z)\setminus\mathbb H$ as well as its representative $\mathfrak z_\ell\in\mathbb H$.
  Petersson constructed a punctured fundamental domain ($\varepsilon > 0$)
\[
\mathcal{F}_{\varepsilon_1,\dots,\varepsilon_r}^*=\mathcal{F}_{(\z_1,\varepsilon_1),\dots, (\z_{r},\varepsilon_r)}^*:=\mathcal{F}^* \Big\backslash \bigcup_{j=1}^{\ell}\mathcal{B}_{\varepsilon_j}\!\left(\z_j\right),
\]
where $\mathcal{F}^*$ is a fundamental domain with $\z_{\ell}$ in the interior of $\Gamma_{\z_{\ell}}\mathcal{F}^*$ and  $\mathcal{B}_{\varepsilon}(\z)$ is the ball around $\z$ of hyperbolic radius $\varepsilon$ (see \eqref{d}). He then defined the \begin{it}regularized inner product between $f$ and $g$\end{it}
\[
\left<f,g\right>:=\lim_{\varepsilon_1, \dots, \varepsilon_r\to 0}\int_{\mathcal{F}_{\varepsilon_1,\dots,\varepsilon_r}^*}f(z) \overline{g(z)} y^{2k} \frac{dx dy}{y^2},
\]
and explicitly determined (see \cite[(6)]{Pe2}) that this regularization converges if and only if for all $n<0$ and $
\varrho \in \H$
\[
c_{f,\varrho}(n)c_{g,\varrho}(n)=0.
\]

\subsection{Proof of Theorem \ref{DualityTheorem}: Duality for meromorphic cusp forms}

\begin{proof}[Proof of Theorem \ref{DualityTheorem}]

The basic idea is to use the fact that, by Proposition \ref{Psiprop}, $\Psi_{2k,n}^{\z}$ is orthogonal to cusp forms for $n<0$ and then compute the inner product in a second way, evaluating it as the sum of two elliptic coefficients. This method was used by Guerzhoy \cite{GuerzhoyGrids} to obtain duality results for Fourier coefficients of weakly holomorphic modular forms.  In order to evaluate the inner product of meromorphic cusp forms as a sum of elliptic coefficients, we mimic calculations given in \cite[Theorem 4.1]{BOR} and \cite{BKvP}.

To begin, for $\z_1,\z_2\in\H$ and $n<0\leq m$, we use Theorem \ref{thm:PoincProperties} to compute, using Stokes' Theorem,
\begin{equation}\label{eqn:inner1}
0=\left<\Psi_{2k,n}^{\z_1},\Psi_{2k,m}^{\z_2}\right> = \!\left(4\y_2\right)^{1-2k}\left<\Psi_{2k,n}^{\z_1},\xi_{2-2k}\left(\mathbb{P}_{2-2k,-m-1}^{\z_2}\right)\right>.
\end{equation}
Rewriting
\[
\overline{\xi_{2-2k}\left(\mathbb{P}_{2-2k,-m-1}^{\z_2}(z)\right)} = y^{-2k} L_{2-2k}\left(\mathbb{P}_{2-2k,-m-1}^{\z_2}\right),
\]
\eqref{eqn:inner1} becomes
\[
\lim_{\varepsilon_{1}, \varepsilon_{2}\to 0^+} \int_{\mathcal{F}_{(\z_1,\varepsilon_1), (\z_{2},\varepsilon_r)}^*} \Psi_{2k,n}^{\z_1}(z)L_{2-2k}\left(\mathbb{P}_{2-2k,-m-1}^{\z_2}(z)\right) \frac{dx dy}{y^2}.
\]
Stokes' Theorem together with invariance of the integrand under the action of $\SL_2(\Z)$ then yields
\begin{multline*}
-\lim_{\varepsilon_{1}\to 0^+}\int_{\partial \mathcal{B}_{\varepsilon_1}(\z_1)\cap \mathcal{F}_{(\z_1,\varepsilon_1),(\z_{2},\varepsilon_2)}^*} \Psi_{2k,n}^{\z_1}(z)\mathbb{P}_{2-2k,-m-1}^{\z_2}(z) dz\\
-\delta_{\left[\z_1\right]\neq\left[\z_2\right]}\lim_{\varepsilon_{2}\to 0^+}\int_{\partial \mathcal{B}_{\varepsilon_2}(\z_2)\cap \mathcal{F}_{(\z_1,\varepsilon_1),(\z_{2},\varepsilon_2)}^*} \Psi_{2k,n}^{\z_1}(z)\mathbb{P}_{2-2k,-m-1}^{\z_2}(z) dz,
\end{multline*}
where $\partial \mathcal{B}_{\varepsilon}(\z)$ denotes the boundary of $\mathcal{B}_{\varepsilon}(\z)$.
The differential $\Psi_{2k,n}^{\z_1}(z)\mathbb{P}_{2-2k,-m-1}^{\z_2}(z) dz$ is invariant under $\Gamma_{\z_j}$, and hence we may extend the integrals to precisely one copy of $\mathcal{B}_{\varepsilon_j}(\z_j)$, obtaining
\begin{equation}\label{eqn:inner2}
0=\frac{1}{\omega_{\z_1}} \mathcal{J}\!\left(\z_1\right) + \frac{\delta_{\z_1\neq \z_2}}{\omega_{\z_2}}\mathcal{J}\!\left(\z_2\right),
\end{equation}
where
\[
\mathcal{J}
(\varrho)
:=\lim_{\varepsilon\to 0^+}\int_{\partial \mathcal{B}_{\varepsilon}
(\varrho)
} \Psi_{2k,n}^{\z_1}(z)\mathbb{P}_{2-2k,-m-1}^{\z_2}(z) dz.
\]
Note that $r_{
\varrho
}(z)=\varepsilon$ for $z\in \partial \mathcal{B}_{\varepsilon}(
\varrho
)$.  Hence, plugging in the elliptic expansions \eqref{eqn:Psiexpand} around $\varrho=\z_j$, of $\Psi_{2k,n}^{\z_1}$ and \eqref{eqn:Pexpand} of $\mathbb{P}_{2-2k,-m-1}^{\z_2}$ we evaluate
\begin{align*}
\mathcal{J}(\varrho)&=\mathcal{C}_{2k-1,m+1}\lim_{\varepsilon\to 0^+}\int_{\partial \mathcal{B}_{\varepsilon}(\varrho)}(z-\overline{\varrho})^{-2}\\
&\qquad \left( 2\omega_{\z_1}\delta_{[\varrho]=[\z_1]} X^n+ \sum_{\ell\geq 0} c_{2k,\varrho}^{\z_1}(n,\ell) X^{\ell}\right)
\Bigg( 2\omega_{\z_2}
\delta_{[\varrho]=[\z_2]} X^{-m-1}\\
&\qquad\quad+ \sum_{\ell\geq 0} c_{2-2k,\varrho}^{\z_2,+}(-m-1,\ell) X^{\ell} + \sum_{\ell<0} c_{2-2k,\varrho}^{\z_2,-}(-m-1,\ell) \beta_0\!\left(1-\varepsilon^2; 2k-1,-\ell\right) X^{\ell}\Bigg)dz,
\end{align*}
where we abbreviate $X=X_{\varrho}(z)$ and $r=r_{\varrho}(z)$.  The integral gives $2\pi i$ times the residue of the integrand at $z=\varrho$, yielding
\begin{multline*}
\mathcal{J}(\varrho)=2\pi i\mathcal{C}_{2k-1,m+1} \bigg( 2\omega_{\z_1}
\delta_{[\varrho]=[\z_1]} c_{2-2k,\varrho}^{\z_2,+}(-m-1,-n-1) + 2\omega_{\z_2}
\delta_{[\varrho]=[\z_2]} c_{2k,\varrho}^{\z_1}(n,m)\\
+ \sum_{\ell\geq 0} c_{2k,\varrho}^{\z_1}(n,\ell)c_{2-2k,\varrho}^{\z_2,-}(-m-1,-\ell-1)	\lim_{\varepsilon\to 0}\beta_0\!\left(1-\varepsilon^2; 2k-1,\ell+1\right)\bigg).
\end{multline*}
However, Lemma \ref{bet0lem} (1) implies that as, $\varepsilon\to 0$,
\[
\beta_0\!\left(1-\varepsilon^2; 2k-1,\ell+1\right)\ll \varepsilon^{2\ell+2}
\]
so that for $\ell+1>0$
\[
\lim_{\varepsilon\to 0}\beta_0\!\left(1-\varepsilon^2; 2k-1,\ell+1\right)=0.
\]
Therefore
\[
\mathcal{J}(\varrho)=4\pi i\mathcal{C}_{2k-1,m+1}\left( \omega_{\z_1}\delta_{
[\varrho]=[\z_1]
}c_{2-2k,\varrho}^{\z_2,+}(-m-1,-n-1) + \omega_{\z_2}
\delta_{[\varrho]=[\z_2]} c_{2k,\varrho}^{\z_1}(n,m)\right).
\]
Plugging back into \eqref{eqn:inner2} yields
\[
4\pi i\mathcal{C}_{2k-1,m+1} \left( c_{2-2k,\z_1}^{\z_2,+}(-m-1,-n-1) +c_{2k,\z_2}^{\z_1}(n,m)\right)=0.
\]
This gives the claim after the change of variables $n\mapsto -n-1$.

\end{proof}
\begin{remark}
The orthogonality to cusp forms shown by Petersson can also be reproven directly either by rewriting $\Psi_{2k,n}^{\z_1}$ as a constant multiple of $\xi_{2-2k}(\mathbb{P}_{2-2k,-n-1}^{\z_1})$ or rewriting $\Psi_{2k,m}^{\z_2}$ as a constant multiple of $D^{2k-1}(\mathbb{P}_{2-2k,m+2k-1})$.

\end{remark}

\end{document}